\documentclass[a4paper,12pt]{amsart}

\usepackage{amsfonts,amssymb,amscd,amsmath,latexsym,amsbsy,enumerate,stmaryrd,a4wide,verbatim,bm}
\usepackage{hyperref}

\theoremstyle{plain}
\newtheorem{thm}{Theorem}[section]
\newtheorem{cor}[thm]{Corollary}

\newtheorem{lem}[thm]{Lemma}
\newtheorem{prop}[thm]{Proposition}
\newtheorem{Def}[thm]{Definition}

\theoremstyle{remark}
\newtheorem{remark}[thm]{Remark}
\numberwithin{equation}{section}

\newcommand{\R}{\mathbb R}

\newcommand{\C}{\mathbb C}
\newcommand{\Z}{\mathbb Z}
\newcommand{\T}{\mathbb T}
\newcommand{\al}{\alpha}

\newcommand{\ga}{\gamma}

\newcommand{\De}{\Delta}
\newcommand{\eps}{\varepsilon}

\newcommand{\te}{\theta}
\newcommand{\vte}{\vartheta}

\newcommand{\la}{\lambda}

\newcommand{\Om}{\Omega}

\newcommand{\tensor}{\otimes}

\newcommand{\rphis}[5]{\,_{#1}\varphi_{#2} \left( \genfrac{.}{.}{0pt}{}{#3}{#4}
	\ ;#5 \right)}
\newcommand{\mvert}{\mkern 2mu | \mkern 2mu}
\newcommand{\U}{\mathcal U}
\newcommand{\su}{\mathfrak{su}}

\newcommand{\cM}{\mathcal M}
\newcommand{\cA}{\mathcal A}
\newcommand{\cB}{\mathcal B}
\newcommand{\cC}{\mathcal C}

\newcommand{\Y}{\mathsf Y}

\begin{document}
	\title{Multivariate Askey-Wilson functions and overlap coefficients}
	\author{Wolter Groenevelt}
	\address{Technische Universiteit Delft, DIAM, PO Box 5031,
		2600 GA Delft, the Netherlands}
	\email{w.g.m.groenevelt@tudelft.nl}

	\maketitle
	
\begin{abstract}
We study certain overlap coefficients appearing in representation theory of the quantum algebra $\U_q(\mathfrak{sl}_2(\C))$. The overlap coefficients can be identified as products of Askey-Wilson functions, leading to an algebraic interpretation of the multivariate Askey-Wilson functions introduced by Geronimo and Iliev \cite{GI}. We use the underlying coalgebra structure to derive $q$-difference equations satisfied by the multivariate Askey-Wilson functions.
\end{abstract}
	
\section{Introduction}
The Askey-Wilson functions are $q$-hypergeometric functions generalising the Askey-Wilson polynomials \cite{AW}. The latter are the polynomials in $x+x^{-1}$ that are eigenfunctions of the Askey-Wilson $q$-difference operator
\begin{equation} \label{eq:AW difference equation}
A(x)(T_q-1) + A(x^{-1})(T_{q^{-1}}-1),
\end{equation}
where $(T_qf)(x)=f(qx)$ and 
\[
A(x)=\frac{(1-ax)(1-bx)(1-cx)(1-dx)}{(1-x^2)(1-qx^2)}.
\]
Ismail and Rahman \cite{IR} obtained explicit eigenfunctions, not necessarily polynomials, of the Askey-Wilson $q$-difference operator. The Askey-Wilson function is a specific nonpolynomial eigenfunction which appears as the kernel in an integral transform due to Koelink and Stokman \cite{KSt} describing the spectral properties of the Askey-Wilson $q$-difference operator. Unitarity of the integral transform implies orthogonality relations for the Askey-Wilson functions first obtained by Suslov \cite{S},\cite{S1}. In \cite{KSt01} it is shown that the Askey-Wilson functions appear as spherical functions on the $\mathrm{SU}(1,1)$ quantum group; other representation theoretic interpretations are e.g.~in double affine Hecke algebras \cite{St2003} and as $6j$-symbols \cite{Gr06}.

Gasper and Rahman \cite{GR05} introduced multivariate extensions of the Askey-Wilson polynomials, defined as nested products of univariate Askey-Wilson polynomials, generalising Tratnik's \cite{Tr} multivariate Wilson polynomials. It was then shown by Iliev \cite{Il} that the multivariate Askey-Wilson polynomials in $N$ variables are eigenfunctions of $N$ independent $q$-difference operators that can be considered as extensions of the Askey-Wilson $q$-difference operator \eqref{eq:AW difference equation}. Moreover, by using a symmetry property the multivariate Askey-Wilson polynomials were shown to be bispectral. These multivariate polynomials also naturally appear in representation theory, see e.g.~\cite{GIV}, \cite{BC}, \cite{Gr21}. Geronimo and Iliev \cite{GI} extended the results from \cite{Il} to the level of Askey-Wilson functions using analytic continuation, leading to $q$-difference equations for multivariate Askey-Wilson functions. In this paper we give a representation theoretic interpretation of these multivariate functions and their $q$-difference equations. 

The organisation of the paper is as follows. In Section \ref{sec:quantum algebra} we introduce the quantum algebra $\mathcal U_q(\mathfrak{sl}_2(\C))$ and the representation of the algebra we will use. In Section \ref{sec:eigenfunctions} we consider eigenfunctions of two different algebra elements and compute their overlap coefficients. These coefficients are given in terms of a $q$-hypergeometric integral which can be identified as a univariate Askey-Wilson function. By writing the action of the algebra on the eigenfunctions as operators acting in the spectral variables, we construct $q$-difference operators for which the overlap coefficients are eigenfunctions. The $q$-difference operators are shown to be the Askey-Wilson $q$-difference operators. The construction of the Askey-Wilson function is similar to Stokman's construction in \cite{St}, but the construction of the related $q$-difference operators is different. Furthermore, the construction immediately gives a symmetry property of the overlap coefficients which leads to bispectrality of the Askey-Wilson functions. We briefly also consider simpler versions of the overlap coefficients, which we use to derive bilateral summation formulas involving $_2\varphi_1$-functions. In Section \ref{sec:multivariate AW} we extend the results from Section \ref{sec:eigenfunctions} to a multivariate setting using the coalgebra structure of $\mathcal U_q(\mathfrak{sl}_2(\C))$. Similar to the construction of multivariate Askey-Wilson polynomials in \cite{Gr21} this results in an interpretation of the multivariate Askey-Wilson functions as overlap coefficients, and leads to a construction of $q$-difference equations for these functions.

\subsection{Notations}
Throughout the paper $q \in (0,1)$ is fixed. We use standard notations for $q$-shifted factorials, $\te$-functions and $q$-hypergeometric functions as in \cite{GR}. In particular, $q$-shifted factorials and theta-functions are defined by
\[
\begin{aligned}
(x;q)_n &= \prod_{j=0}^{n-1} (1-xq^{j}), &&x \in \C, \, n \in \Z_{\geq 0} \cup \{\infty\},\\
\te(x;q) &= (x,q/x;q)_\infty, &&x \in \C^\times,
\end{aligned}
\]	
from which it follows that
\begin{equation} \label{eq:diffeq shifted factorials}
(qx;q)_\infty = \frac{1}{1-x}(x;q)_\infty, \qquad \te(qx;q) = -\frac1x \te(x;q).
\end{equation}
We use the standard shorthand notations
\[
\begin{split}
(x_1,x_2,\ldots,x_k;q)_n &= \prod_{j=1}^k (x_j;q)_n, \\ \te(x_1,x_2,\ldots,x_k;q) &= \prod_{j=1}^k \te(x_j;q).
\end{split}
\]
Moreover, $\pm$-symbols in exponents inside $q$-shifted factorials or theta functions means taking products over all possible combinations of $+$ and $-$ signs, e.g.
\[
(xy^{\pm 1}  z^{\pm 1};q)_\infty = (xyz,xyz^{-1}, xy^{-1}z,xy^{-1}z^{-1};q)_\infty.
\]

\section{The quantum algebra} \label{sec:quantum algebra}
The quantum algebra $\mathcal U_q = \U_q(\mathfrak{sl}_2(\C))$ is the unital, associative, complex algebra generated by $K$, $K^{-1}$, $E$, and $F$, subject to the relations
\begin{gather*}
	K K^{-1} = 1 = K^{-1}K,\\ KE = qEK, \quad KF= q^{-1}FK,\\ EF-FE =\frac{K^2-K^{-2}}{q-q^{-1}}.
\end{gather*}
$\U_q$ has a comultiplication $\De:\U_q \to \U_q \tensor \U_q$ defined
on the generators by
\begin{equation} \label{eq:comult}
	\begin{aligned}
		\De(K) &= K \tensor K,& \De(E)&= K \tensor E + E \tensor K^{-1}, \\
		\De(K^{-1}) &= K^{-1} \tensor K^{-1},&  \De(F) &= K \tensor F + F \tensor K^{-1}.
	\end{aligned}
\end{equation}
We equip $\U_q$ with the $*$-structure $*:\U_q \to \U_q$ defined on the generators by
\[
K^*=K, \quad E^*=-F, \quad F^* = -E, \quad (K^{-1})^* = K^{-1},
\]
which corresponds to the real form $\su(1,1)$ of $\mathfrak{sl}_2(\C)$.

\subsection{Twisted primitive elements}
For $s,u \in \C^\times$ we define two twisted primitive elements $Y_{s,u}$ and $\widetilde Y_{s,u}$ by
\begin{equation} \label{eq:def Y tildeY}
	\begin{split}
		Y_{s,u} &= uq^\frac12  EK - u^{-1}q^{-\frac12} FK + \mu_s(K^2-1),\\
		\widetilde Y_{s,u} &= uq^\frac12 FK^{-1} -u^{-1}q^{-\frac12} EK^{-1} +  \mu_s(K^{-2}-1),
	\end{split}
\end{equation}
where
\[
\mu_s = \frac{ s + s^{-1}}{q^{-1}-q}.
\]
We have 
\[
(Y_{s,u})^*= Y_{\bar s, \bar u^{-1}} \qquad \text{and} \qquad (\widetilde Y_{s,u})^*=\widetilde Y_{\bar s,\bar u^{-1}}.
\]
From \eqref{eq:comult} it follows that
\begin{equation} \label{eq:Delta(Y)}
	\begin{split}
		\De(Y_{s,u}) &= K^2 \tensor Y_{s,u} + Y_{s,u} \tensor 1, \\
		\De(\widetilde Y_{s,u}) &= \widetilde Y_{s,u} \tensor K^{-2} +1 \tensor \widetilde Y_{s,u},
	\end{split}
\end{equation}
so that  $Y_{s,u}$ and $\widetilde Y_{s,u}$ belong to a left, respectively right, coideal of $\U_q$.

\subsection{Representations}
Let $\cM$ be the space of meromorphic functions on $\C^\times$. For $\la,\eps \in \R$ we define a representation $\pi=\pi_{\la,\eps}$ of $\U_q$ on $\cM$ by	
\begin{equation} \label{eq:definition pi}
\begin{split}
	[\pi_{\la,\eps}(K) f](z) & = q^\eps f(qz), \\
	[\pi_{\la,\eps}(E) f](z) & = z\frac{ q^{-\frac12-i\la-\eps} f(z/q) - q^{\frac12+i\la+\eps} f(qz) }{q^{-1} - q}, \\
	[\pi_{\la,\eps}(F) f](z) & = z^{-1}\frac{ q^{-\frac12-i\la+\eps} f(qz) - q^{\frac12+i\la-\eps} f(z/q) }{q^{-1} - q}.
\end{split}
\end{equation}
We define an inner product by 
\[
\langle f,g\rangle = \frac1{2\pi i} \int_\T f(z) g^\star(z)\, \frac{dz}{z},
\]
with $g^\star(z) = \overline{g(\overline {z}^{-1})}$, and where the unit circle $\T$ has positive orientation.
Let $\U_q^1$ be the subspace of $\U_q$ spanned by $1$, $K$, $K^{-1}$, $E$ and $F$. Suppose that $f,g \in \cM$ are analytic on the annulus $\{q \leq |z| \leq q^{-1}\}$, then
\begin{equation} \label{eq:pi(X^*)}
\langle \pi(X)f,g\rangle = \langle f, \pi(X^*) g \rangle, \qquad X \in \U_q^1,
\end{equation}
which follows from Cauchy's theorem to shift the path of integration. For $X=X_1\cdots X_k$ with $X_i \in \U_q^1$ the same property holds for functions $f,g$ that are analytic on the annulus \mbox{$\{q^k \leq |z| \leq q^{-k}\}$}. 
\medskip

The following result, which is proved by direct verification, will be useful later on.
\begin{lem} \label{lem:involution theta}
The assignment
\[
\vartheta(K)=K^{-1}, \quad \vartheta(E)=F, \quad \vartheta(F)=E
\]
extends to an involutive algebra isomorphism and coalgebra anti-isomorphism $\vartheta:\U_q \to \U_q$ satisfying
\begin{itemize}
	\item $\vartheta(Y_{s,u}) = \widetilde Y_{s,u}$;
	\item  $\pi_{\la,\eps}(\vartheta(X)) = r\circ \pi_{\la,-\eps}(X)\circ r$, $X \in \U_q$, where $r:\cM \to \cM$ is the reflection operator defined by $[rf](z)=f(1/z)$.
\end{itemize} 
\end{lem}

To end this section let us introduce some convenient notation. The functions we study later on will depend on (a subset of) the parameters $s$, $u$, $t$, $v$, $\la$ and $\eps$ coming from the twisted primitive elements $Y_{s,u}$ and $\widetilde Y_{t,v}$, and the representation $\pi_{\la,\eps}$. To simplify  notation we let $\al$ be the ordered $6$-tuple  
\[
\al=(s,u,t,v,\la,\eps).
\]
On such $6$-tuples we define an involution $\vte$, which corresponds to the $\U_q$-involution $\vte$, by
\[
\al^\vte = (t,v,s,u,\la,-\eps).
\]
If $f=f_\al$ is a function depending on $\al$, then we denote by $f^\vte$ the same function with $\al$ replaced by $\al^\vte$; $f^\vte = f_{\al^\vte}$. We sometimes use the notation 
\[
\bar \al = (\bar s, \bar u, \bar t, \bar v, \la, \eps).
\]
Note that $(\bar \al)^\vte=\overline{\al^\vte}$.

\section{Overlap coefficients and univariate Askey-Wilson functions} \label{sec:eigenfunctions}
In this section we consider eigenfunctions of $\pi(Y_{s,u})$ and $\pi(\widetilde Y_{t,v})$ and we identify the overlap coefficients between the eigenfunctions with Askey-Wilson functions. We show that $\pi(\widetilde Y_{t,v})$ acts on the eigenfunctions of $\pi(Y_{s,u})$ as a $q$-difference operator in the spectral variable, which leads to a $q$-difference equation satisfied by the overlap coefficients.

\subsection{Eigenfunctions}
We consider eigenfunctions of $\pi(Y_{s,u})$. From the actions of the $\U_q$-generators \eqref{eq:definition pi} and the definition of $Y_{s,u}$ \eqref{eq:def Y tildeY} it follows that, for $f \in \cM$, 
\begin{equation} \label{eq:pi(Y)}
\begin{split}
[\pi(Y_{s,u}) f](z) =\ & q^{2\eps}\frac{ s+ s^{-1}- uzq^{1+i\la} - u^{-1}z^{-1}q^{-1-i\la} }{q^{-1}-q} f(q^2z) \\
&+ \frac{ uzq^{-i\la} + u^{-1} z^{-1} q^{i\la} - s-s^{-1} }{q^{-1}-q} f(z). 
\end{split}
\end{equation}
The eigenvalue equation $\pi(Y_{s,u})f = \mu f$ now becomes a first-order $q^2$-difference equation, for which the eigenfunctions can be determined in terms of $q^2$-shifted factorials. 
\begin{lem} \label{lem:f_x}
Define $f_x=f_{x,\al}$ by
\[
f_x(z) = \frac{(sq^{1-2i\la}x^{\pm 1},usz q^{1+i\la};q^2)_\infty \te(q^{2\eps-i\la}uz/s;q^2) }{(uzq^{-i\la}x^{\pm 1}, sq^{1-i\la}/uz;q^2)_\infty }, \qquad x \in \C^\times,
\]
then $\pi(Y_{s,u})f_x = (\mu_x-\mu_s)f_x$.
\end{lem}
The $z$-independent factor $(sq^{1-2i\la}x^{\pm 1};q^2)_\infty$ is of course not needed for $f_x$ to be an eigenfunction and is only inserted for convenience later on. Note that $f_{x,\al}$ only depends on the four parameters $s,u,\la,\eps$ of the $6$-tuple $\al$.  
\begin{proof}
The eigenvalue equation $\pi(Y_{s,u})f = (\mu_x - \mu_s)f$ is equivalent to
\[
f(q^2z) = \left(-\frac{sq^{i\la-2\eps}}{uz}\right)\frac{ (1-xuzq^{-i\la})(1-uz q^{-i\la}/x)} { (1-suzq^{1+i\la})(1-sq^{-1-i\la}/uz)} f(z).
\]
From \eqref{eq:diffeq shifted factorials} it follows that $f_x$ is a solution.
\end{proof}
Next we compute the action of $\pi(K^{-1})$ on the $\pi(Y_{s,u})$-eigenfunctions $f_x$. It will be convenient to use the notation $f_{x,s}(z)$ for $f_{x,\al}(z)$ to stress its dependence on the parameter $s$. 
\begin{lem} \label{lem:pi(K)f}
	The function $f_{x,s}$ satisfies
	\[
		\begin{split}
			\pi(K^{-1})f_{x,s} &= a^-(x,s) f_{xq,s/q} + a^-(x^{-1},s) f_{x/q,s/q} \\
			& = a^+(x,s) f_{xq,sq} + a^+(x^{-1},s) f_{x/q,sq},
		\end{split}
	\]
	with $a^\pm(x,s)= a_\al^\pm(x)$  given by
	\[
	a^-(x,s) = -\frac{xq^\eps}{s(1-x^2)}, \qquad a^+(x,s) =  \frac{q^{-\eps}(1-sxq^{1-2i\la})(1-sxq^{1+2i\la})}{1-x^2}.
	\]
\end{lem}
\begin{proof}
We consider the function
\[
g_{x,s}(z) = \frac{f_{x,s}(z)}{(sq^{1-2i\la}x^{\pm 1};q^2)_\infty} = \frac{(usz q^{1+i\la};q^2)_\infty \te(q^{2\eps-i\la}uz/s;q^2) }{(uzq^{-i\la}x^{\pm 1}, sq^{1-i\la}/uz;q^2)_\infty }.
\]
This function satisfies
\[
\begin{split}
g_{xq,s/q}(z) &= - \frac{ sq^{1+i\la-2\eps} (1-uzxq^{-1-i\la}) }{uz(1-sq^{-i\la}/uz)} g_{x,s}(z/q),\\
g_{xq,sq}(z) & =\frac{ 1-uzxq^{-1-i\la} }{1-suzq^{i\la}} g_{x,s}(z/q).
\end{split}
\]
From $g_{x,s}=g_{x^{-1},s}$ we obtain similar expressions for $g_{x/q,s/q}$ and $g_{x/q,sq}$ by replacing $x$ by $x^{-1}$. From a direct calculation it then follows that
	\[
	\begin{split}
		q^{1+2i\la-2\eps} g_{x,s}(z/q) & =
		\frac{1-xq^{1+2i\la}/s}{1-x^2} g_{xq,s/q}(z) + \frac{1-q^{1+2i\la}/xs}{1-x^{-2}} g_{x/q,s/q}(z), \\
		g_{x,s}(z/q)&=\frac{1-xsq^{1+2i\la}}{1-x^2} g_{xq,sq}(z) + \frac{1-sq^{1+2i\la}/x}{1-x^{-2}} g_{x/q,sq}(z).
	\end{split}
	\]
	Note that
	\[
	\begin{split}
		g_{xq,s/q}&= -\frac{xq^{1+2i\la}/s}{1-xq^{1+2i\la}/s}  \frac{ f_{xq,s/q} }{(sq^{1-2i\la}x^{\pm 1};q^2)_\infty},\\
		g_{xq,sq}&=(1-sxq^{1-2i\la}) \frac{ f_{xq,sq} }{(sq^{1-2i\la}x^{\pm 1};q^2)_\infty},
	\end{split}
	\]
	then the result follows from $f_x=f_{x^{-1}}$. 
\end{proof}

Combining the two actions of $\pi(K^{-1})$ on $f_{x,s}$ it follows that  $\pi(K^{-2})$ acts as a $q^2$-difference operator in $x$ on $f_{x,s}$.
\begin{cor} \label{cor:pi(K^{-2})f}
	The function $f_x$ satisfies
	\[
	\begin{split}
		\pi(K^{-2})f_{x} &= A(x) f_{xq^2} + B(x) f_{x} + A(x^{-1}) f_{x/q^2},\\
	\end{split}
	\]
	where $A  = A_\al$ and $B=B_\al$ are given by
	\[
	\begin{split}
		A(x) &= -\frac{x(1-sxq^{1-2i\la})(1-sxq^{1+2i\la})}{s(1-x^2)(1-q^2x^2)}, \\
        B(x) & = -A(x)-A(1/x).
	\end{split}
	\]
\end{cor}
\begin{proof}
	This follows from Lemma \ref{lem:pi(K)f} using
\begin{gather*}
A(x) = a^-(xq,sq)a^+(x,s), \\ B(x) = a^-(1/xq,sq) a^+(x,s) + a^-(xq,sq) a^+(1/x,s),
\end{gather*}
and $a^-(1/xq,sq)=-a^-(xq,sq)$.
\end{proof}

Applying Lemma \ref{lem:pi(K)f} twice it follows that $\pi(K^{-2})$ can also be realized as a difference operator acting in $x$ and $s$. This result will be useful in Section \ref{sec:multivariate AW} where we consider difference operators in a multivariate setting.
\begin{cor} \label{cor:pi(K^{-2})f v2}
	The function $f_{x,s}$ satisfies
	\[
	\begin{split}
	\pi(K^{-2}) f_{x,s} & = A^-(x) f_{xq^2,s/q^2} + B^-(x) f_{x,s/q^2} + A^-(x^{-1}) f_{x/q^2,s/q^2} \\
	& = A^+(x) f_{xq^2,sq^2} + B^+(x) f_{x,sq^2} + A^+(x^{-1}) f_{x/q^2,sq^2},
	\end{split}
	\]
	where $A^\pm = A^\pm_\al$ and $B^\pm = B^\pm_\al$ are given by
	\[
	\begin{split}
	A^-(x) &= \frac{ x^2 q^{2+2\eps}}{s^2(1-x^2)(1-q^2x^2)},\\
	B^-(x) &= \frac{q^{2\eps}}{s^2(1-x^{2}/q^2)(1-1/x^2q^2)},
	\end{split}
	\]
	and
	\[
	\begin{split}
		A^+(x) & = \frac{q^{-2\eps}(1-sxq^{1-2i\la})(1-sxq^{3-2i\la})(1-sxq^{1+2i\la})(1-sxq^{3+2i\la})}{(1-x^2)(1-q^2x^2)},\\
		B^+(x) & = \frac{ q^{-2\eps-1}(q^{-1}+q) (1-sxq^{1-2i\la})(1-sq^{1-2i\la}/x)(1-sxq^{1+2i\la})(1-sq^{1+2i\la}/x)} {(1-x^2/q^2)(1-1/x^2q^2)}.
	\end{split}
	\]
\end{cor}
\begin{proof}
	This follows from Lemma \ref{lem:pi(K)f} using
	\[
	A^\pm(x) = a^\pm (x,s)a^\pm(xq,sq^{\pm 1})
	\]
	and
	\[
    \begin{split}
	B^\pm(x) =a^\pm(x,s)a^\pm(1/xq,sq^{\pm 1}) + a^\pm(1/x,s)a^\pm(x/q,sq^{\pm 1}).
    \end{split}
	\]
    Writing this out and simplifying gives the expressions given in the lemma.
\end{proof}

Eigenfunctions of $\pi(\widetilde Y_{t,v})$ can be obtained from the eigenfunctions of $\pi(Y_{s,u})$. 
By applying Lemma \ref{lem:involution theta} to \eqref{eq:pi(Y)}, or by direct verification, it follows that 
\[
\begin{split}
	[\pi(\widetilde Y_{t,v})f](z) =&\  q^{-2\eps}\frac{t+t^{-1}- v^{-1}zq^{-1-i\la} - v z^{-1} q^{1+i\la} }{q^{-1}-q}f(z/q^2)\\ & + \frac{v^{-1}zq^{i\la}+v z^{-1} q^{-i\la}-t-t^{-1}  }{q^{-1}-q} f(z),
\end{split}
\]
which is essentially equation \eqref{eq:pi(Y)} with $(s,u,\eps,z)$ replaced by $(t,v,-\eps, z^{-1})$. Using Lemma \ref{lem:involution theta} it follows that $f_x^{r,\vte} = r f_{x,\al^\vte}$, where $r$ is the reflection operator, is an eigenfunction of $\pi(\widetilde Y_{t,v})$:
\[
\pi_{\la,\eps}(\widetilde Y_{t,v}) (rf_{x,\al^{\vte}}) = ( r \circ  \pi_{\la,-\eps}(Y_{t,v^{-1}})) f_{x,\al^\vte} =  (\mu_x-\mu_t) rf_{x,\al^{\vte}}.
\]
\begin{lem} \label{lem:tildeY f}
	The function 
	\[
	f_x^{r,\vte}(z)   = \frac{ (tq^{1-2i\la}x^{\pm 1}, tv q^{1+i\la}/z;q^2)_\infty \te(vq^{-2\eps-i\la}/tz;q^2) }{(vq^{-i\la}x^{\pm 1}/z, tzq^{1-i\la}/v;q^2)_\infty }
	\]
	satisfies $\pi(\widetilde Y_{t,v})f_x^{r,\vte} = (\mu_x-\mu_t)f_x^{r,\vte}$.
\end{lem}
Slightly abusing notation we will omit $r$ in our notation and write $f_x^{\vte}=f_{x}^{r,\vte}$.

\subsection{Overlap coefficients: Askey-Wilson functions}
Let us write $\bar f_x^{\vte}= rf_{x,\bar \al^\vte}$, then $\bar f_x^{\vte}$ is an eigenfunction of $\pi(\widetilde Y_{\bar t,\bar v})$ for eigenvalue $\mu_{x}-\mu_{\bar t}$. 
\begin{Def} \label{def:Phi}
	Let $0<|ux|,|u/x|, |vy|, |v/y| \leq q^2$ and $0<|s/u|,|t/v|\leq q$. We define $\Phi(x,y)=\Phi_\al(x,y)$ to be the overlap coefficient between $f_x$ and $\bar f_{\bar y}^{\vte}$, i.e.
	\[
	\Phi(x,y) = \left\langle f_x,\bar f_{\bar y}^{\vte} \right\rangle.
	\]
\end{Def}
The following symmetry property of $\Phi$ is immediate from the definition. 
\begin{prop} \label{prop:symmetry}
	$\Phi_\al(x,y) = \overline{ \Phi_{\bar \al^\vte}(\bar y, \bar x) }$
\end{prop}

First we will derive difference equations for $\Phi$. We show that we can write the action of $\widetilde Y_{\bar t,\bar v^{-1}}$ on $f_x$ as a difference operator in $x$. Then the symmetry property of Theorem \ref{prop:symmetry} immediately leads to a corresponding difference operator in $y$. Since we already know the actions of $Y_{s,u}$ and $K^{-2}$ on $f_x$, it suffices to express $\widetilde Y_{t,v^{-1}}$ in terms of $Y_{s,u}$ and $K^{-2}$:
\begin{equation} \label{eq:tilde Y = expression(K,Y)}
	\begin{split}
	\widetilde Y_{t,v^{-1}} = &\ \frac{qu/v-v/uq}{q^{-2}-q^2}K^{-2}Y_{s,u} + \frac{vq/u-u/vq}{q^{-2}-q^2} Y_{s,u}K^{-2} \\&+  \frac{(q^{-1}+q)(t+t^{-1})-(v/u+u/v)(s+s^{-1})}{q^{-2}-q^2}(K^{-2}-1).
\end{split}
\end{equation}
For $v=1$ this is \cite[Lemma 4.3]{Gr21} and the proof, which we have included in the appendix, is the same. It will be convenient to introduce parameters $a,b,c,d$ corresponding to the $6$-tuple $\al$ by
\begin{equation}\label{eq:AW parameters}
\begin{split}
	(a,b,c,d) = (sq^{1-2i\la},sq^{1+2i\la},vtq/u, vq/ut),
\end{split}
\end{equation}

Note that 
\begin{equation} \label{eq:AW parameters dual}
(a^\vte,b^\vte,c^\vte,d^\vte) = (tq^{1-2i\la},tq^{1+2i\la},usq/v, uq/vs).
\end{equation}
\begin{prop} \label{prop:tildeY f_x}
	The function $f_x$ satisfies
	\begin{equation}\label{eq:tildeY f_x q-difference}
		\pi(\widetilde Y_{t,v^{-1}}) f_x = \cA(x)f_{xq^2} +  \cB(x) f_{x} + \cA(x^{-1})   f_{x/q^2},
	\end{equation}
	with $\cA=\cA_\al$ and $\cB=\cB_\al$ given by
	\[
	\begin{split}
		\cA(x) &= \frac{1}{q^{-1}-q}\frac{(1-ax)(1-bx)(1-cx)(1-dx)}{ (q^2/d^\vte)\, (1-x^2)(1-q^2x^2)},\\
		\cB(x) & = \frac{q^2/d^\vte + d^\vte/q^2}{q^{-1}-q} - \mu_t - \cA(x) - \cA(x^{-1}) . 
	\end{split}
	\]
\end{prop}
\begin{proof}
	From \eqref{eq:tilde Y = expression(K,Y)}, Lemma \ref{lem:f_x} and Corollary \ref{cor:pi(K^{-2})f} it follows that $\pi(\widetilde Y_{t,v^{-1}})$ acts as \eqref{eq:tildeY f_x q-difference} on $f_x$ with
	\[
	\begin{split}
		\cA(x) &= \frac{A(x)}{q^{-2}-q^2}\Big((uq/v-v/uq)(\mu_x-\mu_s) +(vq/u-u/vq)(\mu_{xq^2}-\mu_s) \\ 
		&\quad \qquad+(q^{-1}+q)(t+t^{-1})-(v/u+u/v)( s+s^{-1})\Big) \\
		& = -\frac{u}{vxq} \frac{(1-vqx/ut)(1-vqtx/u)}{q^{-1}-q}A(x)
	\end{split}
	\]
	and
	\[
	\begin{split}
		\cB(x) &=  \frac{B(x)}{q^{-2}-q^2}\Big((uq/v-v/uq+vq/u-u/vq)(\mu_x-\mu_s) \\
		& \quad +\frac{B(x)-1}{q^{-2}-q^2}\Big((q^{-1}+q)(t+t^{-1})-(v/u+u/v)(s+s^{-1})\Big)\\
		& = \frac{(t+t^{-1})(B(x)-1)}{q^{-1}-q} + \frac{(u/v+v/u)((s+s^{-1})-(x+x^{-1})B(x) )}{q^{-2}-q^2}.
	\end{split}
	\]
    A calculation shows that
    \[
    \cB(x) = \frac{svq/u+u/svq-t-1/t}{q^{-1}-q}-\cA(x)-\cA(x^{-1}).
    \]
	Rewriting this in terms of the parameters \eqref{eq:AW parameters} proves the proposition.
\end{proof}
Now we are ready to show that the overlap coefficient $\Phi$ is an eigenfunction of the Askey-Wilson difference operator \eqref{eq:AW difference equation}.
\begin{thm} \label{thm:difference eq}
	The overlap coefficient $\Phi(x,y)$ satisfies 
	\[
	\begin{split}
		(\mu_y-\mu_t) \Phi(x,y) &= \cA(x) \Phi(q^2x,y) + \cB(x) \Phi(x,y) + \cA(x^{-1}) \Phi(x/q^2,y),\\
		(\mu_x-\mu_s) \Phi(x,y) &= \cA^\vte(y) \Phi(x,q^2y) + \cB^\vte(y) \Phi(x,y) +\cA^\vte(y^{-1}) \Phi(x,y/q^2),
	\end{split}
	\]
	with $\cA(x)$ and $\cB(x)$ from Proposition \ref{prop:tildeY f_x}.
\end{thm}
\begin{proof} 
	Using Lemma \ref{lem:tildeY f} we obtain
	\[
	(\mu_y - \mu_t) \langle f_x, \bar f^{\vte}_{\bar y} \rangle =  \langle f_x, \pi(\widetilde Y_{\bar t,\bar v}) \bar f^{\vte}_{\bar y} \rangle =  \langle  \pi(\widetilde Y_{t,v^{-1}})f_x, \bar f^\vte_{\bar y} \rangle.
	\]
    Note that this is allowed, since the given conditions ensure that the integrand is analytic on the annulus $\{q^2\leq |z| \leq q^{-2}\}$. By Proposition \ref{prop:tildeY f_x} this is the first stated $q^2$-difference equation.
	The second $q^2$-difference equation follows from the first using the symmetry from Proposition \ref{prop:symmetry} and observing that $\overline{ \cA_{\bar \al^\vte}(\bar y)} = \cA^{\vte}(y)$ and $\overline{\cB_{\bar \al^\vte}(\bar y)} = \cB^\vte(y)$.
\end{proof}

Explicitly $\Phi(x,y)$ is given by the $q$-hypergeometric integral
\begin{equation} \label{eq:Phi integral}
	\begin{split}
\Phi_\al(x,y)  &= (sq^{1-2i\la}x^{\pm1},  tq^{1+2i\la}y^{\pm 1};q^2)_\infty \\ \times \frac{1}{2\pi i}& \int_\cC \frac{ (usz q^{1+i\la},  tvzq^{1-i\la};q^2)_\infty \te(uz q^{2\eps-i\la}/s,vzq^{-2\eps+i \la}/t;q^2) }{(uzq^{-i\la}x^{\pm 1}, sq^{1-i\la}/uz, vzq^{i\la}y^{\pm 1},  tq^{1+i \la}/vz;q^2)_\infty }  \frac{dz}{z},
	\end{split}
\end{equation}
where $\cC=\T$. We can get rid of the conditions on $s,u,t,v,x,y$ from Definition \ref{def:Phi} by replacing the contour $\T$ by a deformation $\cC$ of $\T$ such that the poles $u^{-1}sq^{1-i\la}q^{\Z_{\geq 0}}$, $v^{-1} tq^{1+i\la}q^{\Z_{\geq 0}}$ are inside $\cC$, and the poles $u^{-1}q^{i\la}x^{\pm 1}q^{-\Z_{\geq 0}}$, $v^{-1}q^{-i\la}y^{\pm 1}q^{-\Z_{\geq 0}}$ are outside $\cC$. 

We show that $\Phi$ is a multiple of an Askey-Wilson function \cite{KSt}. For convenience we assume $s,t, u,v \in \R^\times$. In this case we have $a=\overline b$ for the Askey-Wilson parameters \eqref{eq:AW parameters}. Let us now introduce the Askey-Wilson function. Assume that $A,B,C,D$ are parameters satisfying $A=\overline{B}$ and $C,D \in \R$. Define dual parameters $\widetilde A, \widetilde B, \widetilde C, \widetilde D$ by
\begin{equation} \label{eq:dual ABCD}
	\widetilde A = \sqrt{ABCD/q}, \quad \widetilde B=AB/\widetilde A, \quad \widetilde C = AC/\widetilde A, \quad \widetilde D= AD/\widetilde A.
\end{equation}
The Askey-Wilson function with parameters $A,B,C,D$ is defined by
\begin{multline} \label{eq:AW-function = 4phi3}
\psi_\ga(x;A,B,C,D \mvert q)  = \frac{(AB,AC;q)_\infty}{(q/AD;q)_\infty} \rphis{4}{3}{
	Ax,A/x,\widetilde A\ga, \widetilde A/\ga}{AB,AC,AD}{q,q} \\+ \frac{(Ax^{\pm 1}, \widetilde A \ga^{\pm 1}, qB/D, qC/D;q)_\infty}{(qx^{\pm 1}/D, q\ga^{\pm 1}/\widetilde D,AD/q;q)_\infty} \rphis{4}{3}{qx/D, q/Dx, q\ga/\widetilde D,q/\widetilde D \ga}{qB/D, qC/D, q^2/AD}{q,q}.
\end{multline}
 The Askey-Wilson function \eqref{eq:AW-function = 4phi3} is normalized slightly different compared to \cite{KSt}. Using Bailey's transformation \cite[(III.36)]{GR} $\psi_\ga(x)$ can also be written as a multiple of a very-well-poised $_8\varphi_7$-series,
 \[
 \begin{split}
 \psi_\ga(x;A,B,C,D \mvert q) &= \frac{ (AB,AC,BC,Aq/D,qA\ga x^{\pm 1}/\widetilde D;q)_\infty}{ (\widetilde A \widetilde B \widetilde C, q\gamma/\widetilde D, \widetilde Aq/\widetilde D, qx^{\pm 1}/D;q)_\infty}  \\ & \times {}_8W_7\left(\widetilde A\widetilde B\widetilde C \ga/q;Ax,A/x, \widetilde A\ga, \widetilde B\ga, \widetilde C\ga;q, q/\widetilde D\ga \right),
 \end{split}
 \]
 for $|\widetilde D\ga|>q$. In \cite[Section 5.5]{BRSt} it is shown that $\psi_\ga(x)$ can be expressed as a $q$-hypergeometric integral,
\begin{equation} \label{eq:AW integral}
\psi_\ga(x) = \frac{(q,Ax^{\pm 1},\widetilde A\ga^{\pm 1};q)_\infty}{\te(1/\nu,q/AD\nu;q)} \frac{1}{2\pi i} \int_{\cC} \frac{(ABz, ACz;q)_\infty \te(AD\nu z, z/\nu;q)}{(Azx^{\pm 1},\widetilde Az \ga^{\pm 1}, 1/z, q/ADz;q)_\infty} \frac{dz}{z},
\end{equation}
where $\cC$ is a deformation of the unit circle with positive orientation, such that the poles $q^{\Z_{\geq 0}}$ and $(AD/q)q^{\Z_{\geq 0}}$ are inside $\cC$  and excludes the poles $(x^{\pm 1}/A)q^{-\Z_{\geq 0}}$ and $(\ga^{\pm 1}/\widetilde A)q^{-\Z_{\geq 0}}$ are outside $\cC$, and $\nu \in \C^\times$ can be chosen arbitrarily. 
\begin{thm} \label{thm:Phi=AW function}
The overlap coefficient $\Phi(x,y)$ can be expressed as an Askey-Wilson function by
\[
\Phi(x,y) = \frac{ (b^\vte y^{\pm 1} ;q^2)_\infty \te(1/\kappa,ad\kappa;q^2)}{(q^2,q^2 y^{\pm 1}/d^\vte;q^2)_\infty} \psi_y(x;a,b,c,d\mvert q^2),
\]
where 
\begin{equation} \label{eq:AW-parameters + kappa}
(a,b,c,d,\kappa) = (sq^{1-2i\la},sq^{1+2i\la},tvq/u,qv/ut,q^{-1-2\eps+2i\la}).
\end{equation}
\end{thm}
\begin{proof}
In the integral \eqref{eq:Phi integral} for $\Phi$ we substitute $z \mapsto (sq^{1-i\la}/u)w$ then 
\[
\begin{split}
\Phi_\al(x,y) &= (sq^{1-2i\la}x^{\pm1}, tq^{1+2i\la}y^{\pm 1})_\infty \\
& \quad \times \frac1{2\pi i} \int\limits_{(u/sq)\T} \frac{ (s^2 q^2 w, stvq^{2-2i\la}w/u;q^2)_\infty \te(wq^{1+2\eps-2i\la}, svq^{1-2\eps}w/ut)}{(sq^{1-2i\la}wx^{\pm 1}, svqwy^{\pm 1}/u, 1/w, utq^{2i\la}/svw;q^2)_\infty} \frac{dw}{w}.
\end{split}
\]
Then the result follows from comparing this to the integral representation \eqref{eq:AW integral} of the Askey-Wilson function with 
\[
(A,B,C,D,\nu) = (sq^{1-2i\la},sq^{1+2i\la},tvq/u,qv/ut,q^{-1-2\eps+2i\la}) = (a,b,c,d,\kappa),
\]
and using $\widetilde A=suq/v= q^2/d^\vte$.
\end{proof}

\begin{remark} \label{rem:Phi=AW}\*
\begin{enumerate}[(i)]
\item The representation parameter $\eps$ appears only in the parameters $\kappa$, so we see from the expression for $\Phi$ in Theorem \ref{thm:Phi=AW function} that $\eps$ only appears in the multiplicative constant in front of the Askey-Wilson function.
\item The duality property of the Askey-Wilson function as given in \cite[(3.4)]{KSt} states that the Askey-Wilson function is invariant under interchanging the variables $x$ and $\ga$ up to an involution on the parameters $A,B,C,D$;
\begin{equation} \label{eq:AW duality}
\psi_\ga(x;A,B,C,D\mvert q) = \psi_x(\ga;\widetilde A,\widetilde B,\widetilde C,\widetilde D\mvert q).
\end{equation}
The symmetry property of Proposition \ref{prop:symmetry} is very similar, but it is not the same identity. Note that we have
\begin{equation} \label{eq:parameters^vte}
(a^\vte,b^\vte,c^\vte,d^\vte) = (ac/\tilde a, bc/\tilde a, ab/\tilde a, q^2/\tilde a)
\end{equation}
where $\tilde a= \sqrt{abcd/q^2}$.
To obtain the duality \eqref{eq:AW duality} from the identity in Proposition \ref{prop:symmetry} we need to apply the following symmetries of the Askey-Wilson function: 
\begin{equation} \label{eq:AW symmetries}
	\begin{split}
		\psi_\ga(x;A,B,C,D\mvert q) &= \psi_\ga(x;A,C,B,D\mvert q)\\
		& = \frac{(Ax^{\pm 1},\widetilde A \ga^{\pm 1};q)_\infty}{ (qx^{\pm 1}/D,q\ga^{\pm 1}/\widetilde D;q)_\infty}\psi_\ga(x;q/D,B,C,A \mvert q) \\
		& = \frac{(\widetilde C\ga^{\pm 1};q)_\infty }{(q\ga^{\pm 1}/\widetilde D;q)_\infty }\psi_\ga(x;B,A,C,D\mvert q).
	\end{split}
\end{equation}
The first two identities are immediate from \eqref{eq:AW-function = 4phi3}, the third identity is proved in \cite[Proposition 5.27]{BRSt}.
\end{enumerate}
\end{remark}

\subsection{Overlap coefficients: Little $q$-Jacobi functions}
We can also calculate overlap coefficients between the eigenfunctions $f_x$ of $Y_{s,u}$ and eigenfunctions of $K$. The eigenfunction of $K$ are the functions $e_n : \mathcal M \to \C$ given by
\[
e_n(z) = z^n, \qquad n \in \Z.
\]
Using \eqref{eq:definition pi} it is clear that $\pi(K)e_n = q^{n+\eps} e_n$.
\begin{Def} \label{def:phi}
	Let $0<|ux|,|u/x|\leq q^{2}$ and $0<|s/u|\leq q$. For $n \in \Z$ we define $\phi(x,n)=\phi_{\al}(x,n)$ to be the overlap coefficient between $f_{x}$ and $e_n$, i.e.
	\[
	\phi(x;n) = \langle f_{x}, e_n\rangle.
	\]
\end{Def}
The overlap coefficient is given explicitly by
\begin{equation} \label{eq:phi=int}
\phi(x,n) = \frac{1}{2\pi i} \int_\T \frac{ (sq^{1-2i\la}x^{\pm1},usz q^{1+i\la};q^2)_\infty \te(q^{2\eps-i\la}uz/s;q^2) }{(uzq^{-i\la}x^{\pm 1}, sq^{1-i\la}/uz;q^2)_\infty }z^{-n-1} \,dz.
\end{equation}
Note that this is just the Fourier coefficient of $f_x$; $f_x(e^{i\theta}) = \sum_n \phi(x,n) e^{in\theta}$. The next result shows that $\phi(x,n)$ is a multiple of a $_2\varphi_1$-function, which can be recognised as a little $q$-Jacobi function, which is defined by
\[
\varphi_\ga(w;A,B;q) = \rphis{2}{1}{A\ga,A/\ga}{AB}{q,w}, \qquad w \in \C \setminus [1,\infty).
\]
Here we use the one-valued analytic continuation of the $_2\varphi_1$-function, see \cite[\S4.3]{GR}.
\begin{prop} \label{prop:phi=little q-Jacobi}
$\phi(x,n)$ is given in terms of a $_2\varphi_1$-function by
\[
\begin{split}
\phi(x,n)& = \tau^n\frac{(ab;q^2)_\infty \te(1/\kappa;q^2) }{(q^2;q^2)_\infty}  \rphis{2}{1}{ax,a/x}{ab}{q^2,\kappa q^{2-2n}},
\end{split}
\]
with $(a,b,\kappa,\tau) = (sq^{1-2i\la}, sq^{1+2i\la}, q^{-1+2i\la-2\eps}, uq^{i\la-1}/s)$. 
\end{prop}
\begin{proof}
We use the following integral representation of the $_2\varphi_1$-function, see \cite[Section 7]{BR},
\[
\frac{1}{2\pi i} \int_\cC \frac{(t_1z;q)_\infty \te(\mu z/t_2;q) }{(t_2/z,t_3z,t_4z;q)_\infty} \frac{dz}{z} = \frac{(t_1t_2;q)_\infty \te(\mu;q) }{(q,t_2t_3,t_2t_4;q)_\infty} \rphis{2}{1}{t_2t_3,t_2t_4}{t_1t_2}{q,\frac{q}{\mu}},
\]
where $\cC$ is a deformation of the positively oriented unit circle including the poles $t_2 q^{\Z_{\geq 0}}$ and excluding the poles $t_3^{-1}q^{-\Z_{\geq 0}}$ and $t_4^{-1}q^{-\Z_{\geq 0}}$. In \eqref{eq:phi=int} we substitute $z\mapsto (u^{-1}q^{i\la})z$, then we recognize the above integral representation with $q$ replaced by $q^2$ and
\[
(t_1,t_2,t_3,t_4,\mu) = (sq^{1+2i\la},sq^{1-2i\la},x,x^{-1},q^{1-2i\la+2\eps+2n}).
\]
The results then follows from using $\te$-function identities $\te(x;q) = \te(q/x;q)$ and $\te(q^{k}x;q) =  (-x)^{-k} q^{-\frac12k(k-1)}\te(x;q)$, $k \in \Z$. 
\end{proof}
The Fourier expansion $f_x=\sum_n \phi(x,n)e_n$ leads to the following identity, which is a special case of a generating function from \cite[Lemma 3.3]{KR}.
\begin{cor}
Under the conditions of Definition \ref{def:phi},
\[
\sum_{n \in \Z} \rphis{2}{1}{ax,a/x}{ab}{q^2,\kappa q^{2-2n}} t^{-n} = \frac{ (q^2,ax^{\pm 1},ab/t;q^2)_\infty \te(1/t\kappa;q^2)}{(ab,ax^{\pm 1}/t,t;q^2)_\infty \te(1/\kappa;q^2)},
\]
with $a,b,\kappa$ as in Proposition \ref{prop:phi=little q-Jacobi} and $t =1/\tau z= sq^{1-i\la}/uz$.
\end{cor}
Let us also consider the overlap coefficient between $f_x^\vte$ and $re_n=e_{-n}$,
\[
\phi^\vte(x,n) = \langle f_x^\vte,e_{-n}\rangle = \frac{1}{2\pi i} \int_\T \frac{ (tq^{1-2i\la}x^{\pm 1},tvq^{1+i\la}/z;q^2)_\infty \te(vq^{-2\eps-i\la}/tz;q^2) }{(vq^{-i\la}x^{\pm 1}/z, ztq^{1-i\la}/v;q^2)_\infty }z^{n-1} \,dz,
\] 
Using the substitution $z \mapsto z^{-1}$ we immediately see that
\[
\phi^\vte(x,n) = \phi_{\al^\vte}(x,n),
\]
(as was already implied by the notation). So $\phi^\vte$ can be expressed as a $_2\varphi_1$-function by
\[
\begin{split}
	\phi^\vte(x,n)& = (\tau^\vte)^n\frac{(a^\vte b^\vte;q^2)_\infty \te(1/\kappa^\vte;q^2) }{(q^2;q^2)_\infty}  \rphis{2}{1}{a^\vte x,a^\vte/x}{a^\vte b^\vte}{q^2,\kappa^\vte q^{2-2n}} \\
	& = (\tau^\vte)^n\frac{(cq^2/d;q^2)_\infty \te(\bar \kappa q^2;q^2) }{(q^2;q^2)_\infty}  \rphis{2}{1}{ac x/\tilde a,ac/x\tilde a}{cq^2/d}{q^2,q^{-2n}/\bar\kappa}
\end{split}
\]
using \eqref{eq:parameters^vte} and $\kappa^\vte= 1/q^2\bar\kappa$. Parseval's identity $\langle f_x,\bar f_{\bar y}^\vte\rangle = \sum_n \langle f_x,e_n\rangle \langle e_n,\bar f_{\bar y}^\vte\rangle$, which is written in terms of the overlap coefficients as 
\[
\sum_n \phi_\al(x,n) \overline{\phi_{\bar \al^\vte}(\bar y,-n)} = \Phi_\al(x,y),
\] 
then leads to the following summation formula for $_2\varphi_1$-functions.
\begin{cor} 
Under the conditions of Definition \ref{def:Phi} the following summation formula holds:
	\[
	\begin{split}
		\sum_{n \in \Z} &\rphis{2}{1}{ax,a/x}{ab}{q^2,\kappa q^{2-2n}}  \rphis{2}{1}{ac y/\tilde a, ac/y \tilde a}{cq^2/d}{q^2, \frac{q^{2n}}{\kappa}} \left( \frac{q^2}{ad} \right)^n 
		\\
		& \qquad = \frac{(q^2,ac,bcy^{\pm 1}/\tilde a)_\infty\te(ad\kappa;q^2)}{(q^2/ad,cq^2/d,\tilde a y^{\pm 1};q^2)_\infty \te(1/\kappa;q^2)} \rphis{4}{3}{ax,a/x,\tilde a y, \tilde a/ y}{ab,ac,ad}{q^2,q^2} \\& \qquad + \frac{(q^2, q^2b/d, ax^{\pm 1})_\infty \te(ad\kappa;q^2)}{(ab,ad/q^2, q^2x^{\pm 1}/d)_\infty \te(1/\kappa;q^2)} \rphis{4}{3}{q^2x/d, q^2/dx, bc y/\tilde a, bc/y\tilde a}{q^2b/d, q^2c/d, q^4/ad}{q^2,q^2} .
	\end{split}
	\]
where the parameters $a,b,c,d,\kappa$ are given by \eqref{eq:AW-parameters + kappa} and $\tilde a = \sqrt{abcd/q^2}$.
\end{cor}

\section{Overlap coefficients and multivariate Askey-Wilson functions} \label{sec:multivariate AW}
In this section we extend the results from Section \ref{sec:eigenfunctions} to a multivariate setting using the coalgebra structure of $\U_q$. We obtain multivariate Askey-Wilson functions as overlap coefficients for representations of $\U_q^{\tensor N}$. We also show that the multivariate Askey-Wilson functions are simultaneous eigenfunctions of $N$ commuting difference operators coming from commuting  elements in $\U_q^{\tensor N}$. Recall from Remark \ref{rem:Phi=AW} that the representation parameter $\eps$ is not important for studying $q$-difference equations for the Askey-Wilson functions, therefore we choose $\eps=0$ for all representations $\pi_{\la,\eps}$ in this section.

\medskip

Let $N \in \Z_{\geq 2}$ and $\bm{\la} = (\la_1,\ldots, \la_N) \in \R^N$. We consider the representation $\pi_{\bm\la}$ of $\U_q^{\tensor N}$ on $\mathcal M^{\tensor N}$ given by
\[
\pi_{\bm \la} = \pi_{\la_1} \tensor \cdots \tensor \pi_{\la_N},
\]
where $\pi_{\la_j}=\pi_{\la_j,0}$. We use the following notation for iterated coproducts: we define $\De^0$ to be the identity on $\U_q$, and for $n \geq 1$ we define $\De^n:\U_q \to \U_q^{\tensor(n+1)}$ by
\[
\De^n = (\De \tensor 1^{\tensor(n-1)})\circ \De^{n-1},
\]
with the convention $A \tensor B^0=A$.

We consider the following coproducts of the twisted-primitive elements $Y_{s,u}$ and $\widetilde Y_{t,v}$, $s,u,t,v \in \C^\times$: for $j=1,\ldots,N$,
\[
\begin{split}
	\Y_{s,u}^{(j)} &= 1^{\tensor(N-j)} \tensor \De^{j-1}(Y_{s,u}),\\
	\widetilde\Y_{t,v}^{(j)} & = \De^{j-1}(\widetilde Y_{t,v})\tensor 1^{\tensor(N-j)}.
\end{split}
\]
These elements commute, see \cite[Lemma 5.1]{Gr21}: 
for $j,j'=1,\ldots,N$,
	\[
	\Y_{s,u}^{(j)} \Y_{s,u}^{(j')} = 	\Y_{s,u}^{(j')} \Y_{s,u}^{(j)}, \qquad \widetilde \Y_{t,v}^{(j)} \widetilde \Y_{t,v}^{(j')} = 	\widetilde \Y_{t,v}^{(j')} \widetilde \Y_{t,v}^{(j)}.
	\]
\subsection{Eigenfunctions}
Using the eigenfunction $f_{x}$ of $\pi_{\la,0}(Y_{s,u})$ from Lemma \ref{lem:f_x} we obtain simultaneous eigenfunctions of $\pi_{\bm\la}(\Y_{s,u}^{(j)})$, $j=1,\ldots,N$. We use the notation $f_x=f_{x,s,u,\la}$ to stress dependence on the parameters $s,u$ and $\la$.

\begin{lem} \label{lem:pi(Y_j)f}\*
	\begin{enumerate}
	\item For $\bm x, \bm z \in (\C^\times)^{N}$ let $f_{\bm x}=f_{\bm x,\bm \la, s, u} \in \cM^{\tensor N}$ be given by 
	\[
	f_{\bm x}(\bm z) = \prod_{j=1}^N f_{x_j,x_{j+1},u,\la_j}(z_j),
	\]
	with $x_{N+1}=s$, then 
	\[
	\pi_{\bm\la}(\Y_{s,u}^{(j)})f_{\bm x}=(\mu_{x_{N-j+1}}-\mu_s)f_{\bm x}, \qquad j=1,\ldots, N.
	\]
	\item For $\bm y, \bm z \in (\C^\times)^{N}$ let $f_{\bm y}^\vte=f_{\bm y,\bm \la, t, v}^\vte \in \cM^{\tensor N}$ be given by 
	\[
	f_{\bm y}^\vte(\bm z) = \prod_{j=1}^N f_{y_j,y_{j-1},v,\la_j}(1/z_j),
	\]
	with $y_{0}=t$, then 
	\[
	\pi_{\bm\la}(\widetilde{\Y}_{t,v}^{(j)})f_{\bm y}^\vte=(\mu_{y_j}-\mu_t)f_{\bm y}^\vte, \qquad j=1,\ldots, N.
	\]
	\end{enumerate}
\end{lem}
\begin{proof}
	The proof for (ii) follows by induction using the identities
	\[
	\begin{gathered}
	\De^j(\widetilde Y_{t,v})= \De^{j-1}(1) \tensor \widetilde Y_{t,v} + \De^{j-1}(\widetilde Y_{t,v})\tensor K^{-2},\\
	\widetilde Y_{t,v} +(\mu_{y}-\mu_t)K^{-2}= \widetilde Y_{y,v} + (\mu_y-\mu_t)1.
	\end{gathered}
	\]
	See also the proof of \cite[Proposition 5.5]{Gr21} for details.
	The proof of (i) follows after applying the involution $\vte$ and using Lemma \ref{lem:involution theta}. 
\end{proof}

\begin{remark} \label{rem:Y^ij}
	The function $f_{\bm x}$ is an `eigenfunction' of the following $\bm x$-dependent operators:
	For $i = 0,\ldots,N-1$ and $j=1,\ldots,N-i$ we define 
	\[
	\Y_{\bm x,u}^{(i,j)} = 1^{\tensor(N-i-j)} \tensor \De^{j-1}(Y_{x_{N-i+1},u}) \tensor 1^{\tensor i}, 
	\]
	then $\Y_{\bm x,u}^{(i,j)} \Y_{\bm x,u}^{(i,j')} = \Y_{\bm x,u}^{(i,j')} \Y_{\bm x,u}^{(i,j)}$ for $j,j'=1,\ldots,N-i$, and
	\[
	\pi_{\bm \la}(\Y_{\bm x,u}^{(i,j)})f_{\bm x}= (\mu_{x_{N-j-i+1}}-\mu_{x_{N-i+1}})f_{\bm x}.
	\]
	Note that the $\bm x$-dependent element $\Y_{\bm x,u}^{(i,j)}$ depends only on $x_{N-i+1}$. In particular $\Y_{\bm x,u}^{(0,j)}=\Y_{s,u}^{(j)}$ depends on $x_{N+1}=s$, i.e.~it is independent of $\bm x$.
\end{remark}

We define for $j=1,\ldots,N$,
\[
\mathsf K^{-2,(j)} = (K^{-2})^{\tensor j} \tensor 1^{\tensor(N-j)}  \in \U_q^{\tensor N}.
\]
We can express the action of $\mathsf K^{-2,(j)}$ on $f_{\bm x}$ as a $q$-difference operator in $\bm x$. We first need notation for $q$-difference operators. For $i=1,\ldots, N$ we define
\[
[T_i f](\bm x) = f(x_1,\ldots,x_{i-1},x_{i}q^2,x_{i+1},\ldots,x_N),
\]
and for $\bm \nu = (\nu_1,\ldots,\nu_j) \in \{-1,0,1\}^j$ we write
\[
T_{\bm \nu} = T_1^{\nu_1} \cdots T_j^{\nu_j}.
\]
We will use the $q$-difference equations from Corollaries \ref{cor:pi(K^{-2})f} and \ref{cor:pi(K^{-2})f v2} (with $\eps=0$). To stress dependence on the parameters $s$ and $\la$, let us write $A(x)=A_{s,\la}(x)$ and $B(x)=B_{s,\la}(x)$ for the coefficients of the difference equation in Corollary \ref{cor:pi(K^{-2})f}, and similarly $A^\pm_{s,\la}(x)$ and $B^\pm_{s,\la}(x)$ for the coefficients in Corollary \ref{cor:pi(K^{-2})f v2}. Now for $\bm \nu = (\nu_1,\ldots,\nu_j) \in \{-1,0,1\}^j$ we define
\[
A_{\bm \nu,i}(x_{i},x_{i+1}) = 
\begin{cases}
	A^-_{x_{i+1},\la_i}(x_i), & \nu_{i+1}=-1,\\
	A_{x_{i+1},\la_i}(x_i), & \nu_{i+1}=0,\\
	A^+_{x_{i+1},\la_i}(x_i), & \nu_{i+1}=1,
\end{cases}
\]
\[
B_{\bm\nu,i}(x_{i},x_{i+1}) = 
\begin{cases}
	B^-_{x_{i+1},\la_i}(x_i), & \nu_{i+1}=-1,\\
	B_{x_{i+1},\la_i}(x_i), & \nu_{i+1}=0,\\
	B^+_{x_{i+1},\la_i}(x_i), & \nu_{i+1}=1.
\end{cases}
\]
Here we set $\nu_{j+1}=0$ for $\bm\nu \in \{-1,0,1\}^j$.
\begin{prop} \label{prop:K f_x multivariate}
	For $j=1,\ldots,N$, 
	\[
	\pi(\mathsf K^{-2,(j)}) f_{\bm x} = \sum_{\bm \nu \in \{-1,0,1\}^j} C_{\bm \nu}^{(j)}(\bm x) T_{\bm \nu} f_{\bm x},
	\]
	where
	\[
	C_{\bm \nu}^{(j)}(\bm x)= \prod_{i=1}^j C_{\bm \nu,i}^{(j)}(\bm x)
	\]
	with
	\[
	\begin{split}
		C_{\bm \nu,i}^{(j)}(\bm x) = 
		\begin{cases}
			A_{\bm\nu,i}(x_i^{\nu_i},x_{i+1}), & \nu_i \neq 0,\\
			B_{\bm\nu,i}(x_i,x_{i+1}), & \nu_i=0.
		\end{cases}
	\end{split}
	\]
\end{prop}
\begin{proof}
Using Corollary \ref{cor:pi(K^{-2})f} to act with $\pi(K^{-2})$ in the $j$-th factor of $f_{\bm x}$ we see that $\pi_{\bm\la}(\mathsf K^{-2,(j)}) f_{\bm x}$ is equal to
\[
\begin{split}
&\left[\bigotimes_{i=1}^{j-1}\pi_{\la_i}(K^{-2}) \tensor \Big( A_{x_{j+1},\la_j}(x_j)T_j  + B_{x_{j+1},\la_j}(x_j)\mathrm{Id} + A_{x_{j+1},\la_j}(x_j^{-1})T_j^{-1}\Big) \right] f_{\bm x} \\
& \ = \left[\bigotimes_{i=1}^{j-1}\pi_{\la_i}(K^{-2}) \tensor \sum_{\nu_j \in \{-1,0,1\}} C_{\bm \nu,j}^{(j)} T^{\nu_j}_j \right] f_{\bm x}.
\end{split}
\]
Next act with $\pi(K^{-2})$ in the $(j-1)$-th factor of $f_x$ in each term of this sum as follows: apply the $+$-version of Corollary \ref{cor:pi(K^{-2})f v2} if $T_j$ is applied in the $j$-factor of $f_{\bm x}$ and apply the $-$-version if $T_j^{-1}$ is applied in the $j$-th factor; otherwise apply Corollary \ref{cor:pi(K^{-2})f}. Then we see that $\pi(\mathsf K^{-2,(j)}) f_{\bm x}$ is equal to
\[
\begin{split}
	\left[\bigotimes_{i=1}^{j-2}\pi_{\la_i}(K^{-2}) \tensor \sum_{\nu_{j-1},\nu_j \in \{-1,0,1\}} C_{\bm \nu,j-1}^{(j)}(\bm x) C_{\bm \nu,j}^{(j)}(\bm x) T_{j-1}^{\nu_{j-1}} T_{j}^{\nu_j}  \right] f_{\bm x}.
\end{split}
\]
Continuing in this way gives the result.
\end{proof}
Next we determine how $\pi(\widetilde{\mathsf Y}_{t,v}^{(j)})$ acts on $f_{\bm x}$ as a $q$-difference operator in $\bm x$. We need the elements $\Y_{\bm x,u}^{(N-j,j)}$, $j=1,\ldots,N$, from Remark \ref{rem:Y^ij}. Since $\De$ is an algebra homomorphism it follows from the definitions of $\Y_{\bm x,u}^{(N-j,j)}$, $\widetilde \Y_{t,v}^{(j)}$, $\mathsf K^{-2,(j)}$ and \eqref{eq:tilde Y = expression(K,Y)} that for $t,v \in \C^\times$
\begin{equation} \label{eq:tilde Yj=}
	\begin{split}
		\widetilde{\Y}_{t,v^{-1}}^{(j)} = &\ \frac{qu/v-v/uq}{q^{-2}-q^2}\mathsf K^{-2,(j)}\Y_{\bm x,u}^{(N-j,j)} + \frac{vq/u-u/vq}{q^{-2}-q^2} \Y_{\bm x,u}^{(N-j,j)}\mathsf K^{-2,(j)} \\&+  \frac{(q^{-1}+q)(t+t^{-1})-(v/u+u/v)(x_{j+1}+x_{j+1}^{-1})}{q^{-2}-q^2}(\mathsf K^{-2,(j)}-1),
	\end{split}
\end{equation}
for arbitrary $\bm x \in (\C^\times)^N$ and $u \in \C^\times$.
\begin{prop} \label{prop:tilde Y difference operator}
	For $j=1,\ldots,N$, 
	\[
	\pi_{\bm \la}(\widetilde \Y^{(j)}_{t,v^{-1}}) f_{\bm x} = \sum_{\bm \nu \in \{-1,0,1\}^j} \mathcal C_{\bm \nu}^{(j)}(\bm x) T_{\bm \nu} f_{\bm x} + \left(\frac{(u/v+v/u)(x_{j+1}+x_{j+1}^{-1})}{q^{-2}-q^2}-\mu_t\right) f_{\bm x},
	\]
	where
	\[
	\mathcal C_{\bm \nu}^{(j)}(\bm x)=C_{\bm \nu}^{(j)}(\bm x) \times \left\{
	\begin{array}{lr} -\dfrac{ux_1^{-\nu_1}}{vq} \dfrac{(1-vqx_1^{\nu_1}/ut)(1-vqtx_1^{\nu_1}/u)}{q^{-1}-q}, & \nu_1 \neq 0,\\ \\
	\dfrac{(u/v+v/u)(x_{1}+x_{1}^{-1})+(q^{-1}+q)(t+t^{-1})}{q^{-2}-q^2} , & \nu_1 = 0,
	\end{array} \right.
	\]
	with $C_{\bm \nu}^{(j)}$ from Proposition \ref{prop:K f_x multivariate}.
\end{prop}
\begin{proof}
	Note that $\pi_{\bm \la}(\Y_{\bm x,u}^{(N-j,j)})$ acts on $f_{\bm x}$ as multiplication by $\mu_{x_1}-\mu_{x_{j+1}}$, and $\pi_{\bm \la}(K^{-2,(j)})$ acts as a $q$-difference operator in $x_1,\ldots,x_j$ on $f_{\bm x}$. From \eqref{eq:tilde Yj=}, Proposition \ref{prop:K f_x multivariate} and Remark \ref{rem:Y^ij}, we see that $\pi(\widetilde \Y^{(j)}_{t,v^{-1}}) f_{\bm x}$ is a $q$-difference operator as stated in the proposition with coefficients 
	\begin{multline*}
	\mathcal C_{\bm \nu}^{(j)}(\bm x)= \frac{C_{\bm \nu}^{(j)}(\bm x)}{q^{-2}-q^2}
	\Big((qu/v-v/uq)(\mu_{x_1}-\mu_{x_{j+1}})+ (vq/u-u/vq)(\mu_{q^{2\nu_1}x_1}-\mu_{x_{j+1}}) \\  + (q^{-1}+q)(t+t^{-1})-(v/u+u/v)(x_{j+1}+x_{j+1}^{-1}) \Big).
	\end{multline*}
Simplifying gives the result.
\end{proof}

\subsection{Overlap coefficients: multivariate Askey-Wilson functions}
We are now ready to define the overlap coefficient $\Phi(x,y)$ similar as in Definition \ref{def:Phi}. We define an $(N+4)$-tuple $\bm\al$ by
\[
\bm\al= (s,u,t,v,\la_1,\ldots,\la_N).
\]
We define a pairing depending on $\bm \alpha$ and $\bm x,\bm y \in (\C^\times)^N$ by
\[
\langle f,g \rangle = \frac{1}{(2\pi i)^N} \int_{\cC_1}\cdots \int_{\cC_N} f(\bm z) g^\star(\bm z) \frac{dz_N}{z_N}\cdots \frac{ dz_1}{z_1},
\]
where $g^\star(\bm z) = \overline{g(\bar z_1^{-1}, \ldots, \bar z_{N}^{-1})}$, and $\cC_j$ is a deformation of the positively oriented unit circle such that the sequences  $u^{-1}x_{j+1}q^{1-i\la_j}q^{\Z_{\geq 0}}$, $v^{-1} y_{j-1}q^{1+i\la_j}q^{\Z_{\geq 0}}$ are inside $q^2\cC_j$, and the sequences $u^{-1}q^{i\la_j}x_j^{\pm 1}q^{-\Z_{\geq 0}}$, $v^{-1}q^{-i\la_j}y_j^{\pm 1}q^{-\Z_{\geq 0}}$ are outside $q^{-2}\cC_j$. Here $x_{N+1}=s$ and $y_0=t$. Assume that $f(\bm z)$ and $g(\bm z)$ are analytic in $z_j$ on \mbox{$\{z \in q^\theta \cC_j\mid -2 \leq \theta \leq 2\}$} for $j=1,\ldots,N$. Then from applying Cauchy's theorem it follows that this pairing satisfies $\langle \pi_{\bm\la}(X_1X_2)f,g\rangle = \langle f,\pi_{\bm\la}(X_2^*X_1^*)g\rangle$ for $X_1,X_2 \in (\U_q^{1})^{\tensor N}$. 
\begin{Def} For $\bm x,\bm y \in (\C^\times)^N$ we define
	\[
	\Phi_{\bm\al}(\bm x, \bm y) = \langle f_{\bm x}, \bar f_{\bar{\bm y}}^\vte \rangle.
	\]
\end{Def}
It immediately follows from \eqref{eq:Phi integral} that $\Phi_{\bm \al}(\bm x, \bm y)$ can be written as a product of the overlap coefficients $\Phi(x_j,y_j)$ which are essentially Askey-Wilson functions, so $\Phi_{\bm \al}(\bm x, \bm y)$ can be considered as a multivariate Askey-Wilson function.
For this multivariate function we have a symmetry property and difference equations similar to Proposition \ref{prop:symmetry} and Theorem \ref{thm:difference eq}. 
\begin{thm} \label{thm:properties multiAWfunction}
	The overlap coefficient $\Phi(\bm x, \bm y)$ satisfies
	\begin{enumerate}[(i)]
	\item $\displaystyle \Phi_{\bm \al}(\bm x,\bm y) = \prod_{j=1}^N \Phi_{x_{j+1},u,y_{j-1},v,\la,0}(x_j,y_j);$
	\item $\overline{\Phi_{\bar{\bm \al}}(\bar{\bm x},\bar{\bm y})} = \Phi_{\bm \al^\vte}(\hat{\bm y},\hat{\bm x})$, with 
	\[
	\hat{\bm x}=(x_N,\ldots,x_1), \quad \hat{\bm y}=(y_N,\ldots,y_1), \quad \bm \al^\vte = (t,v,s,u,\la_N,\ldots,\la_1);
	\]
	\item For $j=1,\ldots,N$,
	\begin{equation} \label{eq:q-difference Phi}
	\mu_{y_j}\Phi_{\bm \al}(\bm x,\bm y) = \sum_{\bm\nu \in \{-1,0,1\}^j } \cC_{\bm\nu}^{(j)}(\bm x)[T_{\bm\nu} \Phi_{\bm\al}(\, \cdot\,,\bm y)](\bm x)+ \frac{(u/v+v/u)(x_{j+1}+x_{j+1}^{-1})}{q^{-2}-q^2} \Phi_{\bm \al}(\bm x,\bm y),
	\end{equation}
	with $\cC_{\bm \nu}^{(j)}$ from Proposition \ref{prop:tilde Y difference operator}.
	\end{enumerate}
\end{thm}
\begin{proof}
	Identity (ii) follows from the first identity and Theorem \ref{prop:symmetry}, or directly from writing $\Phi_{\bm\al}$ explicitly as an integral. For identity (iii) we assume that the $\bm \al$, $\bm x$ and $\bm y$ are chosen such that $f_{\bm x}(\bm z)$ and $f_{\bm y}^\vte(\bm z)$ are analytic in $z_j$ on $\{ z \in q^\te \cC_j\mid -2 \leq \te \leq 2\}$ for $j=1,\ldots, N$. Then the $q$-difference equations follow from 
	\[
	\langle \pi_{\bm\la}(\widetilde \Y_{t,v}^{(j)})f_{\bm x}, \bar f_{\bar{\bm y}}^\vte\rangle  = \langle f_{\bm x}, \pi_{\bm\la}(\widetilde \Y_{\bar t,\bar v^{-1}}^{(j)}) \bar f_{\bar{\bm y}}^\vte \rangle,
	\]
	using Lemma \ref{lem:pi(Y_j)f} and Proposition \ref{prop:tilde Y difference operator}. The conditions on $\bm \al$, $\bm x$ and $\bm y$ can be removed again by continuity.
\end{proof}

Clearly, combining identities (ii) and (iii) from Theorem \ref{thm:properties multiAWfunction} gives a difference equation in $\bm y$ for $\Phi(\bm x,\bm y)$.

\begin{remark}
The Askey-Wilson algebra encodes the bispectral properties of the Askey-Wilson polynomials. The elements $Y_{t,v}$ and $Y_{s,u}$, together with the Casimir element $\Omega$, generate a copy of the Askey-Wilson algebra in $\U_q$, see \cite{GraZh}. Similarly, for $N=2$ the elements $\widetilde \Y_{t,v}^{(j)}$, $\Y_{s,u}^{(j)}$, $j=1,2$, and $\Delta(\Omega)$, generate a copy of a rank 2 Askey-Wilson algebra \cite{GrW} in $\U_q^{\tensor 2}$. It seems likely that $\widetilde \Y_{t,v}^{(j)}$, $\Y_{s,u}^{(j)}$, $j=1,\ldots,N$, together with appropriate coproducts of $\Om$ generate a copy of the rank $N$ Askey-Wilson algebra \cite{CFPR}.
\end{remark}

\medskip

To end the section, let us summarize the results we have obtained in terms of multivariate Askey-Wilson functions. We set
\[
x_{N+1}=s,\quad y_0=t, \quad \al_0 = v/u, \quad \al_j=q^{2i\la_j} \ \text{for }j=1,\ldots,N,
\]
and write $\bm \al= (y_0,\alpha_0,\alpha_1,\ldots,\alpha_N,x_{N+1})$. 
The multivariate Askey-Wilson functions are given by
\[
\Phi_{\bm\al}(\bm x,\bm y) = \Theta_{\bm\al}(\bm x,\bm y) \prod_{j=1}^N \psi_{y_j}(x_j;qx_{j+1}\al_{j}, qx_{j+1}/\al_j,q\al_0 y_{j-1},q\al_0/y_{j-1} \vert q^2),
\]
with
\[
\Theta_{\bm\al}(\bm x,\bm y) = \prod_{j=1}^N \frac{(q\al_j y_{j-1} y_{j}^{\pm 1};q^2)_\infty \te(q/\al_j,q\al_0x_{j+1}/y_{j-1};q^2) }{(q^2,q x_{j+1} y_j^{\pm 1}/\al_0;q^2)_\infty}.
\]
\*
\textbf{Symmetry property}: $\Phi$ satisfies \[
\overline{\Phi_{\bar{\bm \al}}(\bar{\bm x},\bar{\bm y})} = \Phi_{\hat{\bm \al}}(\hat{\bm y},\hat{\bm x}),
\]
where $\hat{\bm \al} = (x_{N+1},\al_0^{-1},\al_N,\ldots,\al_1,y_0)$.\\

\textbf{$q$-Difference equations}: $\Phi$ satisfies 
\[
\begin{split}
\frac{y_j+y_j^{-1}}{q^{-1}-q}&\Phi_{\bm \al}(\bm x,\bm y)= \\ & \sum_{\bm\nu \in \{-1,0,1\}^j } \cC_{\bm\nu}^{(j)}(\bm x)[T_{\bm\nu} \Phi_{\bm\al}(\, \cdot\,,\bm y)](\bm x) +\frac{(\al_0+\al_0^{-1})(x_{j+1}+x_{j+1}^{-1})}{q^{-2}-q^2} \Phi_{\bm \al}(\bm x,\bm y),
\end{split}
\]
where the coefficients $\cC_{\bm\nu}^{(j)}(\bm x)$ are given explicitly by
\[
\cC_{\bm \nu}^{(j)}(\bm x)= \prod_{i=1}^j C_{\bm \nu,i}^{(j)}(\bm x) \times 
\begin{cases}
-\dfrac{x_1^{-\nu_1}}{\al_0 q} \dfrac{(1-q\al_0x_1^{\nu_1}y_0^{\pm 1})}{q^{-1}-q}, & \nu_1 \neq 0, \\ \\
\dfrac{(\al_0+\al_0^{-1})(x_1+x_1^{-1})+(q^{-1}+q)(t+t^{-1})}{q^{-2}-q^2}, & \nu_1 = 0,
\end{cases}
\]
with
\[
C_{\bm \nu,i}^{(j)}(\bm x) = 
\begin{cases}
	\dfrac{x_{i}^{2\nu_i}}{x_{i+1}^2(1-x_i^{2\nu_i})(1-q^2x_i^{2\nu_i})}, & \nu_i\neq0,\ \nu_{i+1}=-1,\\ \\
	-\dfrac{x_i^{\nu_i}(1-qx_i^{\nu_i} x_{i+1}\al_i^{\pm 1})}{x_{i+1}(1-x_i^{2\nu_i})(1-q^2x_i^{2\nu_i})}, & \nu_i\neq0,\ \nu_{i+1}=0,\\ \\
	\dfrac{(1-qx_{i}^{\nu_i}x_{i+1}\al_i^{\pm 1})(1-q^3x_{i}^{\nu_i}x_{i+1}\al_i^{\pm 1})}{(1-x_i^{2\nu_i})(1-q^2x_i^{2\nu_i})}, & \nu_i\neq0,\ \nu_{i+1}=1,\\ \\
	\dfrac{1}{x_{i+1}^2(1-x_i^{\pm 2}/q^2)}, & \nu_i=0,\ \nu_{i+1}=-1,\\ \\
	\dfrac{x_i(1-qx_i x_{i+1}\al_i^{\pm 1})}{x_{i+1}(1-x_i^2)(1-q^2x_i^2)}+\dfrac{(1-qx_i^{-1} x_{i+1}\al_i^{\pm 1})}{x_ix_{i+1}(1-x_i^{-2})(1-q^2x_i^{-2})}, & \nu_i=0,\ \nu_{i+1}=0,\\ \\
	\dfrac{(1+q^{-2})(1-q\al_i^{\pm 1}x_{i+1} x_i^{\pm 1})}{(1-x_i^{\pm 2}/q^2)}, & \nu_i=0,\ \nu_{i+1}=1.
\end{cases}
\]
Recall here that we use the convention $\nu_{j+1}=0$ for $\bm \nu \in \{-1,0,1\}^j$.

\section{Appendix}
We prove \eqref{eq:tilde Y = expression(K,Y)}: 
\[
	\begin{split}
		\widetilde Y_{t,v^{-1}} = &\ \frac{qu/v-v/uq}{q^{-2}-q^2}K^{-2}Y_{s,u} + \frac{vq/u-u/vq}{q^{-2}-q^2} Y_{s,u}K^{-2} \\&+  \frac{(q^{-1}+q)(t+t^{-1})-(v/u+u/v)(s+s^{-1})}{q^{-2}-q^2}(K^{-2}-1).
	\end{split}
\]
The proof runs along the same lines as in \cite[Lemma 4.3]{Gr21}. 
We define $S,T \in \U_q$ by
\[
S=K^{-2}Y_{s,u} + \mu_s(K^{-2}-1), \qquad T = \frac{K^{-2}Y_{s,u}-Y_{s,u}K^{-2}}{q^{-1}-q},
\]
then by the definition of $Y_{s,u}$ \eqref{eq:def Y tildeY} it follows that
\[
S = uq^{-\frac32}EK^{-1}-u^{-1}q^{\frac32} FK^{-1}, \qquad T= uq^{-\frac12}EK^{-1}+q^\frac12 u^{-1}FK^{-1}.
\]
Using the definition \eqref{eq:def Y tildeY} of $\widetilde Y_{t,v^{-1}}$ we obtain
\[
\widetilde Y_{t,v^{-1}}= \frac{u}{v}\frac{q^{-1}T-S}{q+q^{-1}} -\frac{v}{u} \frac{ S+qT}{q+q^{-1}}+\mu_t(K^{-2}-1),
\]
then expressing $S$ and $T$ in terms of $Y_{s,u}$ and $K^{-2}$ gives the desired expression. 

\section*{Competing interests}
The author has no competing interests to declare that are relevant to the content of this article.

\end{document}